\documentclass[10pt]{amsart}
\usepackage{a4wide}
\usepackage{amsfonts}
\usepackage{psfrag}
\usepackage{bbm}

\usepackage{amsmath}
\usepackage{amssymb}
\usepackage{amsthm}
\usepackage{graphicx}

\newtheorem {lemma}{Lemma}
\newtheorem {thm}{Theorem}
\newtheorem {rem}{Remark}
\usepackage{color}

\newtheorem {cor}{Corollary}
 
\newcommand{\E}{\mathbb{E}}
\newcommand{\N}{\mathbb{N}}

\newcommand{\Q}{\mathbb{Q}}
\DeclareMathOperator{\skel}{skel}
\newcommand{\ind}{\mathbbm{1}}
\newcommand{\Fg}{\mathcal{F}}

\renewcommand{\P}{\mathbb{P}}
\numberwithin{equation}{section}
\numberwithin{equation}{section}
\numberwithin{equation}{section}
\newcommand{\Z}{\mathbb{Z}}

\begin{document}
\title{Catalytic branching processes via spine techniques and renewal theory}
\author{Leif D\"oring}
\address{Fondation Math\'ematique de Paris\newline
 and Laboratoire de Probabilit\'es et Mod\'eles Al\'eatoires (CNRS UMR. 7599) Universit\'e Paris 6 � Pierre et Marie Curie, U.F.R. Math\'ematiques, 4 place Jussieu, 75252 Paris Cedex 05, France}
\email{leif.doering@upmc.fr}
\author{Matthew Roberts}
\address{Weierstrass Institute for Applied Analysis and Stochastics, Mohrenstrasse 39, 10117 Berlin, Germany}
\email{mattiroberts@gmail.com}
\thanks{LD acknowledges the support of the Fondation Sciences Math\'ematiques de Paris. MR thanks
ANR MADCOF (grant ANR-08-BLAN-0220-01) and WIAS for their support.}

\maketitle

\begin{abstract}
	In this article we contribute to the moment analysis of branching processes in catalytic media. The many-to-few lemma based on the spine technique is used  to derive a system of (discrete space) partial differential equations for the number of particles in a variation of constants formulation. The long-time behavior is then deduced from renewal theorems and induction.
\end{abstract}

\section{Introduction and Results}
	A classical subject of probability theory is the analysis of branching processes in discrete or continuous time, going back to the study of extinction of family names by Francis Galton. There have been many contributions to the area since, and we present here an application of a recent development in the probabilistic theory. We identify qualitatively different regimes for the longtime behaviour for moments of sizes of populations in a simple model of a branching Markov process in a catalytic environment.

	\smallskip
	To give some background for the branching mechanism, we recall the discrete-time Galton-Watson process. Given a random variable $X$ with law $\mu$ taking values in $\N$, the branching mechanism is modelled as follows: 	for a deterministic or random initial number $Z_0\in \mathbb{N}$ of particles, one defines for $n=1,2,...$
	\begin{align*}
		Z_{n+1}=\sum_{r=0}^{Z_{n}} X_r(n),
	\end{align*}
	where all $X_r(n)$ are independent and distributed according to $\mu$. Each particle in generation $n$ is thought of as giving birth to a random number of particles according to $\mu$, and these particles together form generation $n+1$. For the continuous-time analogue each particle carries an independent exponential clock of rate 1 and performs its breeding event once its clock rings.
	
	\smallskip	
	It is well-known that a crucial quantity appearing in the analysis is $m=\E[X]$, the expected number of offspring particles. The process has positive chance of long-term survival if and only if $m>1$. This is known as the supercritical case. The cases $m=1$ (critical) and $m<1$ (subcritical) also show qualitatively different behaviour in the rate at which the probability of survival decays. As this paper deals with the moment analysis of a spatial 	relative to this system, we mention the classical trichotomy for the moment asymptotics of Galton-Watson processes:
	\begin{align}
		\lim_{t\to\infty}e^{- k(m-1) t }\E\big[Z_t^k\big]\in(0,\infty)\quad\forall k\in \N &\quad\text{ if } m<1, \label{eqn:m1}\\
		\lim_{t\to\infty}t^{k-1}\E\big[Z_t^k\big]\in(0,\infty)\quad\forall k\in \N &\quad\text{ if } m=1 \label{eqn:m2}\\
		\lim_{t\to\infty}e^{- k(m-1) t }\E\big[Z_t^k\big]\in(0,\infty)\quad\forall k\in \N &\quad\text{ if }m>1 \label{eqn:m3}
	\end{align}
	so that all moments increase exponentially to infinity if $m>1$, increase polynomially if $m=1$, and decay exponentially fast to zero if $m<1$.
	\smallskip
	
	In the present article we are interested in a simple spatial version of the Galton-Watson process for which a system of branching particles moves in space and particles branch only in the presence of a catalyst. More precisely, we start a particle $\xi$ which moves on some countable set $S$ according to a continuous-time Markov process with Q-matrix $\mathcal A$. This particle carries an exponential clock of rate $1$ that only ticks if $\xi$ is at the same site as the catalyst, which we assume sits at some fixed site $0\in S$. 
	 If and when the clock rings, then the particle dies and is replaced in its position by a random number of offspring. This number is distributed according to some offspring distribution $\mu$, and all newly born particles behave as independent copies of their parent: they move on $S$ according to $\mathcal A$ and branch after an exponential rate 1 amount of time spent at $0$.
	 \smallskip
	 
	 In recent years several authors have studied such branching systems. Often the first quantities that are analyzed are moments of the form
	\begin{align*}
		M^k(t,x,y) = \E\big[N_t(y)^k\,\big |\,\xi_0=x\big]\qquad\text{and}\qquad M^k(t,x) = \E\big[N_t^k\,\big |\,\xi_0=x\big],
	\end{align*}
	where $N_t(y)$ is the number of particles alive at site $y$ at time $t$, and $N_t=\sum_{y\in S}N_t(y)$ is the total number of particles alive at time $t$. Under the additional assumption that $\mathcal A=\Delta$ is the discrete Laplacian on $\Z^d$, the moment analysis was first carried out in \cite{ABY98}, \cite{ABY98b},  \cite{AB00} via partial differential equations and Tauberian theorems. More recently, the moment analysis, and moreover the study of conditional limit theorems, was pushed forward to more general spatial movement $\mathcal A$ assuming
		\begin{itemize}
		\item[\textbf{(A1)}] irreducibility,
		\item[\textbf{(A2)}] spatial homogeneity,
		\item[\textbf{(A3})] symmetry,
		\item[\textbf{(A4})] finite variance of jump sizes.
	\end{itemize}
	Techniques such as Bellman-Harris branching processes (see \cite{TV03},\cite{VT05}, \cite{B10}, \cite{B11}), operator theory (see \cite{Y10})  and renewal theory (see \cite{HVT10}) have been applied successfully. Some of these tools also apply in a non-symmetric framework. We present a purely stochastic approach avoiding the assumptions $\textbf{(A1)-(A4)}$. In order to avoid many pathological special cases we only assume 
	\begin{align*}
		\text{\textbf{(A)} } \text{ the motion governed by }\mathcal A \text{ is irreducible}.
	\end{align*}
	This assumption is not necessary, and the interested reader may easily reconstruct the additional cases from our proofs.\\
	 In order to analyze the moments $M^k$ one can proceed in two steps. First, a set of partial differential equations for $M^k$ is derived. This can be done for instance as in \cite{AB00} via analytic arguments from partial differential equations for the generating functions $\E_x[e^{- z N_t(y)}]$ and $\E_x[e^{- z N_t}]$ combined with Fa\`a di Bruno's formula of differentiation. The asymptotic properties of solutions to those differential equations are then analyzed in a second step where more information on the transition probabilities corresponding to $\mathcal A$ implies more precise results on the asymptotics for $M^k$. This is where the finite variance assumption is used via the local central limit theorem.
	 \smallskip
		
	The approach presented in this article is based on the combinatorial spine representation of \cite{HM11} to derive sets of partial differential equations, in variation of constants form, for the $k$th moments of $N_t(y)$ and $N_t$. A set of combinatorial factors can be given a direct probabilistic explanation, whereas the same factors appear otherwise from Fa\`a di Bruno's formula. Those equations are then analyzed via renewal theorems. We have to emphasize that under the assumption \textbf{(A)} only, general precise asymptotic results are of course not possible so that we aim at giving a qualitative description. Compared to the fine results in the presence of a local central limit theorem (such as Lemma 3.1 of \cite{HVT10} for finite variance transitions on $\Z^4$) our qualitative description is rather poor. On the other hand, the generality of our results allows for some interesting applications. For example, one can easily deduce asymptotics for moments of the number of particles when the catalyst is not fixed at zero, but rather follows some Markov process of its own, simply by considering the difference walk.
	\smallskip
	
	 To state our main result, we denote the transition 	probabilities of $\mathcal A$ by $p_t(x,y)=\P_x(\xi_t=y)$ and the Green function by
	\begin{align*}
		G_\infty(x,y)=\int_0^\infty p_t(x,y)\,dt.
	\end{align*}
	Recall that, by irreducibility, the Green function is finite for all $x,y\in S$ if and only if $\mathcal A$ is transient. For the statement of the result let us further denote by 
	\begin{align*}
		L_t(y)=\int_0^t \ind_{\{\xi_s=y\}} ds
	\end{align*}
	the time of $\xi$ spent at site $y$ up to time $t$.

		\begin{thm}\label{thm:1}
		Suppose that $\mu$ has finite moments of all orders; then the following regimes occur for all integers $k\geq 1$:
		\begin{itemize}
			\item[i)] If the branching mechanism is \textbf{subcritical}, then
				\begin{align*}
					\lim_{t\to\infty}M^k(t,x)\in (0,\infty) \quad\text{ if } \mathcal A \text{ is transient},\\
					\lim_{t\to\infty}M^k(t,x)= 0 \quad \text{ if } \mathcal A \text{ is recurrent},
				\end{align*}
				and
				\begin{align*}
					\lim_{t\to\infty}M^k(t,x,y) =0 &\qquad \text{ in all cases.}
				\end{align*}
			\item[ii)] If the branching mechanism is \textbf{critical}, then 
								\begin{align*}
		   	\lim_{t\to\infty}\frac{M^k(t,x)}{\E_x[L_t(0)^{k-1}]}&\in (0,\infty)\quad\text{and}\quad M^1(t,x,y)=p_t(x,y).
				\end{align*}
				
			\item[iii)] If the branching mechanism is \textbf{supercritical}, then there is a critical constant
				\begin{align*}
					\beta=\frac{1}{G_\infty(0,0)}+1\geq 1
				\end{align*}
				such that
				\begin{itemize}
					\item[a)] for $m<\beta$
						\begin{align*}
						\lim_{t\to\infty}M^1(t,x)&\in (0,\infty)\quad\text{and}\quad \lim_{t\to\infty}M^1(t,x,y)=0;
						\end{align*}
						further, there exist constants $c$ and $C$ such that
						\[c\E_x[L_t(0)^{k-1}]\leq M^k(t,x) \leq C t^{k-1}.\]
					\item[b)] for $m=\beta$
						\begin{align*}
							\lim_{t\to\infty}M^k(t,x)=\infty,
\end{align*}
and
						\begin{align*}					
							\lim_{t\to\infty}M^k(t,x,y)=\infty\quad &\text{if }\int_0^\infty rp_r(0,0)\,dr=\infty\\
												\lim_{t\to\infty}{M^k(t,x,y)\over t^{k-1}}\in (0,\infty)\quad &\text{if }\int_0^\infty rp_r(0,0)\,dr<\infty.
						\end{align*}
					(In both cases the growth is subexonential.)
					
					\item[c)] for $m>\beta$
				\begin{align*}
									\lim_{t\to\infty} e^{-kr(m) t}M^k(t,x,y)\in (0,\infty)
\quad\text{and}\quad	\lim_{t\to\infty} e^{-kr(m) t}M^k(t,x)&\in (0,\infty)
				\end{align*}
				where $r(m)$ equals the unique solution $\lambda$ to $\int_0^\infty e^{-\lambda t}p_t(0,0)\,dt =\frac{1}{m-1}$.
				\end{itemize}
		\end{itemize}	
	\end{thm}
We did not state all the asymptotics in cases ii) and iii)a). Our methods, see Lemma \ref{highermoments}, do allow for investigation of these cases too; in particular they show how $M^k(t,x,y)$ can be expressed recursively by $M^i(t,x,y)$ for $i<k$. However, without further knowledge of the underlying motion, it is not possible to give any useful and general information. If more information on the tail of $p_t(x,y)$ is available then the recursive equations can indeed be analyzed: for instance for kernels on $\Z^d$ with second moments the local central limit theorem can be applied leading to $p_t(x,y)\sim Ct^{-d/2}$, and such cases have already been addressed by other authors.\\

	
		The formulation of the theorem does not include the limiting constants. Indeed, the proofs give some of those (in an explicit form involving the transition probabilities $p_t$) in the supercritical regime but they seem to be of little use. The use of spectral theory for 			symmetric $Q$-matrices $\mathcal A$ allows one to derive the exponential growth rate $r(m)$ as the maximal eigenvalue of a Schr\"odinger operator with one-point potential and the appearing constants via the eigenfunctions. 		Our renewal theorem based proof gives the representation of $r(m)$ as the inverse of the Laplace transform of $p_t(0,0)$ at $1/(m-1)$ and the eigenfunction expressed via integrals of $p_t(0,0)$. As $p_t(0,0)$ is rarely 			known explicitly, the integral form of the constants is not very useful (apart from the trivial case of Example 1 below).
		Only in case iii) b) for $\int_0^\infty rp_r(0,0)\,dr=\infty$ are the proofs unable to give strong asymptotics. This is caused by the use of an infinite-mean renewal theorem which only gives asymptotic bounds up to an unknown factor 		between $1$ and $2$.
		There is basically one example in which $p_t(0,0)$ is trivially known:
	
	\smallskip
	\textbf{Example 1:} For the trivial motion $\mathcal A=0$, i.e. branching particles are fixed at the same site as the catalyst, the supercritical cases iii) a) and b) do not occur as $\mathcal A$ is trivially recurrent so that $\beta=1$. Furthermore, in this example 	$p_t(0,0)=1$ for all $t\geq 0$ so that $r(m)=m-1$. In fact by examining the proof of Theorem \ref{thm:1} one recovers (\ref{eqn:m1},\ref{eqn:m2},\ref{eqn:m3}) with all constants.
	
	\medskip		
	The explicit representation for the exponential growth rate allows for a more careful comparison with the non-spatial case. 
	\begin{cor}
		Let $r(m)$ be the exponential growth rate obtained in the supercritical case of Theorem~1. Then $$m \mapsto r(m)$$ is convex, with
		\begin{align*}
			r(1)=0,\qquad r(m)\leq m-1,\qquad \lim_{m\to\infty}\frac{r(m)}{(m-1)}=1.
		\end{align*}
	\end{cor}
	\begin{proof}
		This follows from elementary manipulations of the defining equation for $r(m)$.
	\end{proof}
	
\begin{rem}
	We reiterate here that our results can be generalized when the fixed branching source is replaced by a random branching source moving according to a random walk independent of the branching particles. For the proofs the branching particles only have to be replaced by branching particles relative to the branching source.
\end{rem}

\section{Proofs}

    The key tool in our proofs will be the many-to-few lemma proved in \cite{HM11} which relies on modern spine techniques. These emerged from work of Kurtz, Lyons, Pemantle and Peres in the mid-1990s \cite{kurtz_et_al:conceptual_kesten_stigum, lyons:simple_path_to_biggins, lyons_et_al:conceptual_llogl_mean_behaviour_bps}. The idea is that to understand certain functionals of branching processes, it is enough to carefully study the behaviour of one special particle, the \emph{spine}. In particular very general \emph{many-to-one} lemmas emerged, allowing one to easily calculate expectations of sums over particles like
    \[\E\left[\sum_{v\in N_t} f(v) \right],\]
    where $f(v)$ is some well-behaved functional of the behaviour of the particle $v$ up to time $t$, and $N_t$ here is viewed as the \emph{set} of particles alive at time $t$, rather than the number. It will always be clear from the 	context which meaning for $N_t$ is intended.
    
	It is natural to ask whether similar results exist for higher moments of sums over $N_t$. This is the idea behind \cite{HM11}, wherein it turns out that to understand the $k$th moment one must consider a system of $k$ particles. 	The $k$ particles introduce complications compared to the single particle required for first moments, but this is still significantly simpler than controlling the behaviour of the potentially huge random number of particles in
	$N_t$.
	
	\vspace{2mm}
    
    While we do not need to understand the full spine setup here, we shall require some explanation.
    
    For each $k\geq0$ let $p_k=\P(X = k)$ and $m_k = \E[X^k]$, the $k$th moment of the offspring distribution (in particular $m_1 = m$). We define a new measure $\Q = \Q^k_x$, under which there are $k$ distinguished lines of descent known as spines. The construction of $\Q$ relies on a carefully chosen change of measure, but we do not need to understand the full construction and instead refer to \cite{HM11}. In order to use the technique, we simply have to understand the dynamics of the system under $\Q$. Under $\Q^k_x$ particles behave as follows:
    \begin{itemize}
    \item We begin with one particle at position $x$ which (as well as its position) has a mark $k$. We think of a particle with mark $j$ as carrying $j$ spines.
    \item Whenever a particle with mark $j$, $j\geq1$, spends an (independent) exponential time with parameter $m_j$ in the same position as the catalyst, it dies and is replaced by a random number of new particles with law $A_j$.
    \item The probability of the event $\{ A_j=a\}$ is $a^j p_a m_j^{-1}$. (This is the \emph{$j$th size-biased} distribution relative to $\mu$.)
    \item Given that $a$ particles $v_1,\ldots,v_a$ are born, the $j$ spines each choose a particle to follow independently and uniformly at random. Thus particle $v_i$ has mark $l$ with probability $a^{-l}(1-a^{-1})^{j-l}$, $l=0,\ldots,j$, $i=1,\ldots,a$. We also note that this means that there are always $k$ spines amongst the particles alive; equivalently the sum of the marks over all particles alive always equals $k$.
    \item Particles with mark 0 are no longer of interest (in fact they behave just as under $\P$, branching at rate 1 when in the same position as the catalyst and giving birth to numbers of particles with law $\mu$, but we will not need to use this).
    \end{itemize}
    For a particle $v$, we let $X_v(t)$ be its position at time $t$ and $B_v$ be its mark (the number of spines it is carrying). Let $\sigma_v$ be the time of its birth and $\tau_v$ the time of its death, and define $\sigma_v(t) = \sigma_v \wedge t$ and $\tau_v(t) = \tau_v \wedge t$. Let $\chi^i_t$ be the current position of the $i$th spine. We call the collection of particles that have carried at least one spine up to time $t$ the \emph{skeleton} at time $t$, and write $\skel(t)$. Figure \ref{skelfig} gives an impression of the skeleton at the start of the process.
    
   \begin{figure}[h!]
   \centering
   \includegraphics[width=9cm]{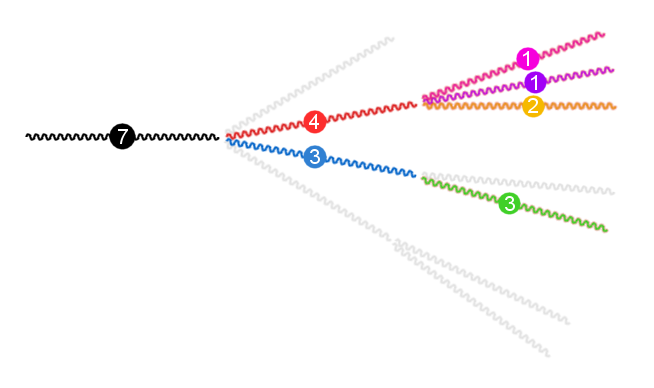}
   \caption{An impression of the start of the process: each particle in the skeleton is a different colour, and particles not in the skeleton are drawn in pale grey. The circles show the number of spines being carried by each particle in the skeleton. \label{skelfig}}
   \end{figure}
   
  A much more general form of the following lemma was proved in \cite{HM11}.
    
    \begin{lemma}[Many-to-few]\label{many_to_few}
    Suppose that $f:\mathbb{R}\to\mathbb{R}$ is measurable. Then, for any $k\geq1$,
    \begin{multline*}
    \E\left[\sum_{v_1,\dots,v_k \in N_t} f(X_{v_1}(t))\cdots f(X_{v_k}(t))\right]\\
    = \Q^k\left[ f(\chi^1_t)\cdots f(\chi^k_t) \prod_{v\in \skel(t)} \exp\left( \left(m_{B_v} - 1\right)\int_{\sigma_v(t)}^{\tau_v(t)} \ind_{0}(X_v(s))ds \right)\right].
    \end{multline*}
    \end{lemma}
    
	Clearly if we take $f\equiv 1$, then the left hand side is simply the $k$th moment of the number of particles alive at time $t$. The lemma is useful since the right-hand side depends on at most $k$ particles at a time, rather than the arbitrarily large random number of particles on the left-hand side.
	
	\smallskip
	Having introduced the spine technique, we can now proceed with the proof of Theorem 1.
	We first use Lemma \ref{many_to_few} for the case $k=1$, which is simply the many-to-one lemma, to deduce two convenient representations for the first moments: a Feynman-Kac expression 	and a variation of constants 	formula. Indeed, the exponential expression equally works  for other random potentials and, hence, is well known for instance in the parabolic Anderson model literature. More interestingly, the variation of constants representation is most useful in 	the case of  a one-point potential: it simplifies to a renewal type equation. Understanding when those are proper renewal equations replaces the spectral theoretic arguments of \cite{ABY98} and explains the different cases appearing in 			Theorem 1.
	\begin{lemma}\label{l}
		The first moments can be expressed as
		\begin{align}
			M^1(t,x)&=\E_x\big[e^{(m-1)\int_0^t \ind_{0}(\xi_r)\,dr}\big],\label{eqn:s2}\\
			M^1(t,x,y)&=\E_x\big[e^{(m-1) \int_0^t \ind_{0}(\xi_r)\,dr}\ind_{y}(\xi_t)\big], \label{eqn:s1}
		\end{align}
		where $\xi_t$ is a single particle moving with Q-matrix $\mathcal A$. Furthermore, these quantities fulfill
		\begin{align}
			M^1(t,x)&=1+(m-1) p_t(x,0) \ast M^1(t,0),\label{eqn:2}\\
			M^1(t,x,y)&=p_t(x,y)+(m-1) p_t(x,0) \ast M^1(t,0,y),\label{eqn:1}
		\end{align}
		where $\ast$ denotes ordinary convolution in $t$.
	\end{lemma}
	For completeness we include a proof of these well-known relations. First let us briefly mention why the renewal type equations occur naturally. The Feynman-Kac representation can be proved in various ways; we derive it simply from the many-to-few lemma. The Feynman-Kac formula 	then leads naturally to solutions of discrete-space heat equations with one-point potential:
	\begin{align*}
		\begin{cases}
			\frac{\partial }{\partial  t}u(t,x)=\mathcal A u(t,x)+(m-1)\ind_{0}(x)u(t,x)\\
			u(0,x)=\ind_y(x)
		\end{cases}.
	\end{align*}
	Applying the variation of constants formula for solutions gives
	\begin{align*}
		u(t,x)&=P_t u(0,x)+\int_0^t P_{t-s} (m-1)\ind_0(x)u(s,x)\,ds\\
		&=p_t(x,y)+(m-1)\int_0^t p_{t-s}(x,0) u(s,x)\,ds,
	\end{align*}
	where $P_t$ is the semigroup corresponding to $\mathcal A$, i.e. $P_tf(x) = \E_x[f(\xi_t)]$.

	\begin{proof}[Proof of Lemma \ref{l}]
		To prove (\ref{eqn:s2}) and (\ref{eqn:s1}) we apply the easiest case of Lemma \ref{many_to_few}: we choose $k=1$ and $f\equiv 1$  (resp.  $f(z) = \ind_y(z)$ for (\ref{eqn:s1})). Since there is exactly one spine at all times, the 		skeleton reduces to a single line of descent. Hence $m_{B_v}-1=m-1$ and the integrals in the product combine to become a single integral along the path of the spine up to time $t$. Thus
		\[M^1(t,x)=\Q_x\big[e^{(m-1)\int_0^t \ind_{0}(\xi_r)\,dr}\big] \hspace{5mm}\hbox{ and }\hspace{5mm} M^1(t,x,y)=\Q_x\big[e^{(m-1) \int_0^t \ind_{0}(\xi_r)\,dr}\ind_{y}(\xi_t)\big]\]
		which is what we claimed but with expectations taken under $\Q$ rather than the original measure $\P$. However we note that the motion of the single spine is the same (it has Q-matrix $\mathcal A$) under both $\P$ and 		$\Q$, so we may simply replace $\Q$ with $\P$, giving (\ref{eqn:s2}) and (\ref{eqn:s1}).
		
	The variation of constants formulas can now be derived from the Feynman-Kac formulas. We only prove the second identity, as the first can be proved similarly. We use the exponential series to get
	\begin{align*}
		&\quad\E_x\Big[e^{(m-1) \int_0^t \ind_{0}(\xi_r)\,dr}\ind_{y}(\xi_t)\Big]\\
		&=\E_x\left[\sum_{n=0}^{\infty}\frac{(m-1)^n}{n!}\left(\int_0^t \ind_{0}(\xi_{r})\,dr\right)^n\ind_{y}(\xi_t)\right]\\
		&=\P_x(\xi_t=y)+\E_x\left[\sum_{n=1}^{\infty}\frac{(m-1)^n}{n!}\int_0^t \dots \int_0^t \ind_{0}(\xi_{r_1})\cdots \ind_{0}(\xi_{r_n})\,dr_n\ldots dr_1\ind_{y}(\xi_t)\right]\\
		&=p_t(x,y)+\E_x\left[\sum_{n=1}^{\infty}(m-1)^n\int_0^t \int_{r_1}^{t} \dots \int_{r_{n-1}}^{t} \ind_{0}(\xi_{r_1})\cdots \ind_{0}(\xi_{r_n})\,dr_n\ldots d r_2 dr_1\ind_{y}(\xi_t)\right].
	\end{align*}
	The last step is justified by the fact that the function that is integrated is symmetric in all arguments and, thus, it suffices to integrate over a simplex. We can exchange sum and expectation and obtain that the last expression equals
	\begin{eqnarray*}
		p_t(x,y)+ (m-1) \int_0^t \sum_{n=1}^{\infty}(m-1)^{n-1} \int_{r_1}^{t} \dots \int_{r_{n-1}}^{t} \P_x[ \xi_{r_1}=0,\ldots, \xi_{r_n}=0]\,dr_n\ldots d r_2 dr_1.
	\end{eqnarray*}
	Due to the Markov property, the last expression equals
	\begin{align*}
		p_t(x,y)+ (m-1) \int_0^t p_{r_1}(x,0) \sum_{n=1}^{\infty}(m-1)^{n-1} \int_{r_1}^{t} \dots \int_{r_{n-1}}^{t} \P_0 [ \xi_{r_2-r_1}=0, \ldots, \xi_{r_n-r_1}=0] \,dr_n\ldots d r_2 dr_1
	\end{align*}
	and can be rewritten as
	\begin{align*}
		p_t(x,y)+ (m-1) \int_0^t p_{r_1}(x,0) \left(\sum_{n=1}^{\infty}(m-1)^{n-1} \int_{0}^{t-r_1} \dots \int_{r_{n-1}}^{t-r_1} \P_0[ \xi_{r_2}=0, \ldots, \xi_{r_n}=0] \,dr_n\ldots d r_2 \right) dr_1.
	\end{align*}
	Using the same line of arguments backwards for the term in parentheses, the assertion follows.
\end{proof}
	
	Having derived variation of constants formulas, there are different ways to analyze the asymptotics of the first moments. Assuming more regularity for the transition probablities, this can be done as sketched in the next remark.
	
	\begin{rem}\label{rem}
		Taking Laplace transforms $\mathcal L$ in $t$, one can transform (\ref{eqn:2}), and similarly (\ref{eqn:1}), into the algebraic equation
		\begin{align*}
				\mathcal L M^1(\lambda,x)&=\frac{1}{\lambda}+(m-1) \mathcal L M^1(\lambda,0)\mathcal L p_\lambda (x,0)\qquad ,\lambda>0,
		\end{align*}
		which can be solved explicitly to obtain
		\begin{align}
			\mathcal L M^1(\lambda,x)=\frac{1}{\lambda(1-(m-1)\mathcal L p_{\lambda}(x,0))}\qquad ,\lambda>0.\label{eqn:3}
		\end{align}
		Assuming the asymptotics of $p_t(x,0)$ are known for $t$ tending to infinity (and are sufficiently regular), the asymptotics of $\mathcal L p_\lambda(x,0)$ for $\lambda$ tending to zero can be deduced from Tauberian 			theorems. Hence, from Equation (\ref{eqn:3}) one can then deduce the asymptotics of $\mathcal L M^1(\lambda,x)$ as $\lambda$ tends to zero. This, using Tauberian theorems in the reverse direction, allows one to deduce 			the asymptotics of $M^1(t,x)$ for $t$ tending to infinity.
		
		Unfortunately, to make this approach work, ultimate monotonicity and asymptotics of the type $p_t(x,0)\sim Ct^{-\alpha}$ are needed. This motivated the authors of \cite{ABY98} to assume $\textbf{(A4)}$ so that by the local 	central limit theorem
	\begin{align*}
		p_t(x,0)\sim \left(\frac{d}{2\pi}\right)^{d/2} t^{-d/2}.
	\end{align*}
	\end{rem}
	As we did not assume any regularity for $p_t$, the aforementioned approach fails in general. We instead use an approach based on renewal theorems recently seen in \cite{DS10}.
	
\begin{proof}[Proof of Theorem \ref{thm:1} for $M^1$]
		Taking into account irreducibility and the Markov property of $\mathcal A$, we see that the property ``$\int_0^\infty \ind_{0}(\xi_r)\,dr=\infty$ almost surely'' does not depend on the starting value $\xi_0$. 
		To prove case i), we simply apply dominated convergence to (\ref{eqn:s2}) and (\ref{eqn:s1}). If $\mathcal A$ is transient, then $\int_0^\infty \ind_{0}(\xi_r)\,dr<\infty$ almost surely and $M^1(t,x)$ converges to a constant. On the other hand if $\mathcal A$ is recurrent, then $\int_0^\infty \ind_{0}(\xi_r)\,dr=\infty$ almost surely and $M^1(t,x)\to 0$. In both cases $M^1(t,x,y)\to0$, because if $\mathcal A$ is transient then $\ind_{\{\xi_t=y\}}\to 0$ almost surely, and if $\mathcal A$ is recurrent then $M^1(t,x,y)\leq M^1(t,x)\to0$.\\
		Regime ii) is trivial as here $M^1(t,x)=1$ and $M^1(t,x,y)=p_t(x,y)$. Next, for regime iii) a) we exploit both the standard and the reverse H\"older inequality for $p>1$:
		\begin{align}
			 M^1(t,x,y)&\geq \E_x\big[e^{-(1/(p-1))(m-1) \int_0^t \ind_{0}(\xi_r)\,dr}\big]^{-(p-1)}p_t(x,y)^p,\label{eqn:11}\\
			M^1(t,x,y)&\leq \E_x\big[e^{p(m-1) \int_0^t \ind_{0}(\xi_r)\,dr}\big]^{1/p}p_t(x,y)^{(p-1)/p}.\label{eqn:12}
		\end{align}
		In the recurrent case $G_\infty(0,0)=\infty$ and thus $\beta=1$, so this case has already been dealt with in regime ii). Hence we may assume that $\mathcal A$ is transient so that $\int_0^\infty\ind_{0}(\xi_r)\,dr<\infty$ with 		positive probability. This shows that the expectation in the lower bound (\ref{eqn:11}) converges to a finite constant. By assumption $m-1<\beta$ so that there is $p>1$ satisfying 			$p(m-1)<\beta$. With this choice 		of $p$,  part 3) of Theorem 1 of \cite{DS10} implies that also the expectation in the upper bound (\ref{eqn:12}) converges to a finite constant. In total this shows that
		\begin{align*}
			C p_t(x,y)^p\leq M^1(t,x,y) \leq C' p_t(x,y)^{(p-1)/p}
		\end{align*}
		and the claim for $M^1(t,x,y)$ follows. For $M^1(t,x)$ we can directly refer to Theorem 1 of \cite{DS10}.\\
		
		For regimes iii) b) and c) we give arguments based on renewal theorems. A closer look at the variation of constants formula (\ref{eqn:1}) shows that only for $x=0$, $M^1(t,x,y)$ occurs on both sides of the equation. Hence, 		we start with the case $x=0$ and afterwards deduce the asymptotics for $x\neq 0$.

		Let us begin with the simpler case iii) c).  As mentioned above, in this case we may assume that $\mathcal A$ is transient so that 
		$\int_0^\infty p_r(0,0)\,dr<\infty$. Hence, dominated convergence ensures that the equation $\int_0^\infty e^{-\lambda t}p_t(0,0)\,dt=1/(m-1)$ has a unique positive root $\lambda$, which we call $r(m)$. The definition of 			$r(m)$ shows that
		$U(dt):=(m-1) e^{-r(m) t}p_t(0,0)\,dt$ is a probability measure on $[0,\infty)$ and furthermore $e^{-r(m) t}p_t(0,y)$ is directly Riemann integrable. Hence the classical renewal theorem (see page 349 of \cite{F71}) can be 			applied to the (complete) renewal equation	
		\begin{align*}
			f(t)=g(t)+f\ast U (t),
		\end{align*}
		with $f(t)=e^{-r(m) t}M^1(t,0,y)$ and $g(t)=e^{-r(m) t}p_t(0,y)$. The renewal theorem implies that
		\begin{align}
			\lim_{t\to\infty}f(t)=\frac{\int_0^\infty g(s)\,ds}{\int_0^\infty U((s,\infty))\,ds}\in (0,\infty)\label{eqn:h}
		\end{align}
		so that the claim for $M^1(t,0,y)$ follows including the limiting constants.
		
		\medskip
		For iii) b), we need to be more careful as the criticality implies that $(m-1)\int_0^{\infty}p_r(0,0)\,dr=1$. Hence, the measure $U$ as defined above is already a probability measure so that the variation of constants formula is indeed a proper renewal equation. The renewal measure $U$ only has finite mean if additionally
		\begin{align}
			\int_0^\infty rp_r(0,0)\,dr<\infty.\label{eqn:hh}
		\end{align}	
		In the case of finite mean the claim follows as above from (\ref{eqn:h}) without the exponential correction (i.e. $r(m)=0$). Note that $p_t(0,y)$ is directly Riemann integrable as the case $\beta>0$ implies that 		
		$\mathcal A$ is transient and $p_t(0,y)$ is decreasing.\\
		If (\ref{eqn:hh}) fails, we need a renewal theorem for infinite mean variables. Iterating Equation (\ref{eqn:1}) reveals the representation
		\begin{align}
			M^1(t,0,y)=p_{t}(0,y)\ast \sum_{n\geq 0}(m-1)^np_t(0,0)^{\ast n},\label{eqn:iter}
		\end{align}
		where $\ast n$ denotes $n$-fold convolution in $t$ and $p_t(0,y)\ast p_t(0,0)^{\ast 0}=p_t(0,y)$. Note that convergence of the series is justified by
		\begin{align*}
			(m-1)^np_t(0,0)^{\ast n}\leq \left((m-1)\int_0^t p_r(0,0)\,dr\right)^n
		\end{align*}
		and the assumption on $m$. Lemma 1 of \cite{E73} now implies that
		\begin{align}\label{1}
			\sum_{n\geq 0}(m-1)^np_t(0,0)^{\ast n}\approx  \frac{t}{(m-1) \int_0^{t}\int_s^\infty p_r(0,0)\,drds}
		\end{align}
		which tends to infinity as $(m-1)\int_s^\infty p_r(0,0)\,dr\to 0$ for $s\to\infty$ since we assumed that $(m-1)p_r(0,0)$ is a probability density in $r$. To derive from this observation the result for $M^1(t,0,y)$, note that the 			simple bound $p_t(0,y)\leq 1$ gives the upper bound
		\begin{align}\label{eqn:a1}
			M^1(t,0,y)\leq \int_0^t \sum_{n\geq 0}(m-1)^np_r(0,0)^{\ast n}\,dr.
		\end{align}
		For a lower bound, we use that due to irreducibility and continuity of $p_t(0,y)$ in $t$, there are $0<t_0<t_1$ and $\epsilon>0$ such that $p_t(0,y)>\epsilon$ for $t_0\leq t\leq t_1$. This shows that							\begin{align}\label{eqn:a2}
			M^1(t,0,y)\geq \epsilon \int_{t-t_0}^{t-t_1}\sum_{n\geq 0}(m-1)^np_r(0,0)^{\ast n}\,dr.
		\end{align}
		Combined with (\ref{1}) the lower and upper bounds directly prove the claim for $M^1(t,0,y)$.
		
		\smallskip
		It remains to deal with regime iii) b) and c) for $x\neq 0$. The results follow from the asymptotics of the convolutions as those do not vanish at infinity. But this can be deduced from simple upper and lower bounds  similar to 		(\ref{eqn:a1}) and (\ref{eqn:a2}). 
		
		\smallskip
		The asymptotic results for the expected total number of particles $M^1(t,x)$ follow from similar ideas: estimating as before
		\begin{align*}
			1+\epsilon \int_{t-t_0}^{t-t_1}M^1(r,0)\,dr\leq M^1(t,x)\leq 1+\int_0^t M^1(r,0)\,dr,
		\end{align*}
		and applying case 2) of Theorem 1 of \cite{DS10} to (\ref{eqn:s2}) with $x=0$, the result follows.
		\end{proof}
		
		
		
	
	
	We now come to the crucial lemma of our paper. We use the the many-to-few lemma to reduce higher moments of $N_t$ and $N_t(y)$ to the first moment. More precisely, a system of equations is derived that can be	solved inductively once the first moment is known. This particular useful form is caused by the one-point catalyst. A similar system can be derived in the same manner in the deterministic case if the one-point potential is replaced by a $n$-point potential. However the case of a random $n$-point potential is much more delicate as the sources are ``attracted'' to the particles, destroying any chance of a renewal theory approach.

	\begin{lemma}\label{highermoments}
		For $k\geq 2$ the $k$th moments fulfill
	\begin{align}
			M^k(t,x)&=M^1(t,x)+M^1(t,x,0)\ast g_k\big((M^1(t,0),\cdots, M^{k-1}(t,0)\big),\label{eqn:111}\\
			M^k(t,x,y)&=M^1(t,x,y)+M^1(t,x,0)\ast g_k\big(M^1(t,0,y),\cdots, M^{k-1}(t,0,y)\big),\label{eqn:222}
		\end{align}
		where
		\[g_k\big(M^1,...,M^{k-1}\big)=\sum_{j=2}^k\E\left[ X \choose j \right]\sum_{\substack{i_1,\ldots,i_j > 0 \\ i_1 + \ldots + i_j = k}}\frac{k!}{i_1!\cdots i_j!}M^{i_1}\cdots M^{i_j}.\]
	\end{lemma}
	\begin{proof}
		We shall only prove equation (\ref{eqn:111}); the proof of equation (\ref{eqn:222}) is almost identical. We recall the spine setup and introduce some more notation. To begin with, all $k$ spines are carried by the same 			particle $\xi$ which branches at rate $m_k = \E[X^k]$ when at 0. Thus the $k$ spines separate into two or more particles at rate $m_k - m$ when at 0 (since it is possible that at a birth event all $k$ spines continue to follow the same particle, which 		happens at rate $m$). We consider what happens at this first ``separation'' time, and call it $T$. 
		
		Let $i_1, \ldots, i_j >0$, $i_1 + \ldots + i_j = k$, and define $A_k(j;i_1,\ldots,i_j)$ to be the event that at a separation event, $i_1$ spines follow one particle, $i_2$ follow another, \ldots, and $i_j$ follow another. The first 			particle splits into $a$ new particles with probability $a^k p_a m_k^{-1}$ (see the definition of $\Q^k$). Then given that the first particle splits into $a$ new particles, the probability that $i_1$ spines follow one particle, $i_2$ 		follow another, \ldots, and $i_j$ follow another is
		\[\frac{1}{a^k} \cdot {a\choose j} \cdot \frac{k!}{i_1!\cdots i_j!}\]
		(the first factor is the probability of each spine making a particular choice from the $a$ available; the second is the number of ways of choosing the $j$ particles to assign the spines to; and the third is the number of ways of 		rearranging the spines amongst those $j$ particles). Thus the probability of the event $A_k(j;i_1,\ldots,i_j)$ under $\Q^k$ is
		\[\frac{1}{m_k} \E\left[ X\choose j\right] \frac{k!}{i_1!\cdots i_j!}.\]
		(Note that, as expected, this means that the total rate at which a separation event occurs is
		\[ m_k \cdot \frac{1}{m_k} \sum_{j=2}^k \E\left[ X\choose j\right] \sum_{\substack{i_1,\ldots,i_j > 0 \\ i_1 + \ldots + i_j = k}} \frac{k!}{i_1!\cdots i_j!} = m_k - m\]
		since the double sum is just the expected number of ways of assigning $k$ things to $X$ boxes without assigning them all to the same box.)
		
		However, for $j\geq2$, \emph{given} that we have a separation event, $A_k(j;i_1,\ldots,i_j)$ occurs with probability
		\[\frac{1}{m_k} \E\left[ X\choose j\right] \frac{k!}{i_1!\cdots i_j!}\left(\frac{m_k}{m_k - m}\right).\]
		
		Write $\chi_t$ for the position of the particle carrying the $k$ spines for $t\in[0,T)$, and define $\Fg_t$ to be the filtration containing all information (including about the spines) up to time $t$. Recall that the skeleton $\skel(t)$ 		is the tree generated by particles containing at least one spine up to time $t$; let $\skel(s;t)$ similarly be the part of the skeleton falling between times $s$ and $t$. Using the many-to-few lemma with $f=1$, the fact that 			by definition before $T$ all spines sit on the same particle and integrating out $T$, we obtain
		\begin{align*}
		\E\big[N_t^k\big] &= \Q^k\left[\prod_{v\in\skel(t)} e^{(m_{B_v}-1)\int_{\sigma_v(t)}^{\tau_v(t)}\ind_0(X_v(s))ds}\right]\\
		&= \Q^k\left[ e^{(m_k-1)\int_0^T \ind_0(\chi_s)ds} \ind_{\{T\leq t\}} \Q^k\left[\left.\prod_{v\in\skel(T;\hspace{0.5mm}t)} e^{(m_{B_v}-1)\int_{\sigma_v(t)}^{\tau_v(t)}\ind_0(X_v(s))ds}\right|\Fg_T\right]\right]\\
		& \hspace{50mm} + \Q^k\left[ e^{(m_k-1)\int_0^t \ind_0(\chi_s)ds} \ind_{\{T > t\}}\right]\\
		&= \int_0^t \Q^k\Bigg[ e^{(m_k-1)\int_0^u \ind_0(\chi_s)ds} (m_k-m)\ind_{0}(\chi_u) e^{-(m_k-m)\int_0^u \ind_0(\chi_s)ds}\\
		& \hspace{25mm}\cdot\Q^k\bigg[\prod_{v\in\skel(u;\hspace{0.5mm}t)} e^{(m_{B_v}-1)\int_{\sigma_v(t)}^{\tau_v(t)}\ind_0(X_v(s))ds}\bigg|\Fg_u; T=u\bigg]\Bigg] du\\
		& \hspace{50mm} + \Q^k\left[ e^{(m_k-1)\int_0^t \ind_0(\chi_s)ds} e^{-(m_k-m)\int_0^t \ind_0(\chi_s)ds}\right].
		\end{align*}
		To prove equation (\ref{eqn:222}), the same arguments are used with $f = \ind_y$ in place of $f = 1$.
		Now we split the sample space according to the distribution of the numbers of spines in the skeleton at time $T$. Since, given their positions and marks at 		time $T$, the particles in the skeleton behave independently, we may split the product up into $j$ independent factors. Thus
		\begin{align*}
		\E\big[N_t^k\big] &= \int_0^t \sum_{j=2}^k \sum_{\substack{i_1,\ldots,i_j > 0 \\ i_1 + \ldots + i_j = k}} \E\left[ X\choose j\right] \frac{k!}{i_1!\cdots i_j!} \Q^k\Bigg[ e^{(m-1)\int_0^u \ind_0(\chi_s)ds} \ind_{0}(\chi_u)\\
		& \hspace{25mm}\cdot\prod_{l=1}^j\Q^{i_l}\bigg[\prod_{v\in\skel(t-u)} e^{(m_{B_v}-1)\int_{\sigma_v(t-u)}^{\tau_v(t-u)}\ind_0(X_v(s))ds}\bigg]\Bigg] du\\
		& \hspace{50mm} + \Q^k\left[ e^{(m-1)\int_0^t \ind_0(\chi_s)ds} \right]\\
		&= \int_0^t \sum_{j=2}^k \sum_{\substack{i_1,\ldots,i_j > 0 \\ i_1 + \ldots + i_j = k}} \E\left[ X\choose j\right] \frac{k!}{i_1!\cdots i_j!} \E_x\left[ N_u(0) \right] \cdot\prod_{l=1}^j\E_0\left[ N_{t-u}^{i_l} \right] du + \E_x\left[ N_t \right],
		\end{align*}
		where we have used the many-to-few lemma backwards with $f = \ind_0$ (first expectation) and $f = 1$ (two last expectations) to obtain the last line. This is exactly the desired equation (\ref{eqn:111}). For Equation 			(\ref{eqn:222}) we again use $f = \ind_y$ in place of $f = 1$ and copy the same lines of arguments.
	\end{proof}
	
	\begin{rem}
		The factors appearing in $g_k$ are derived combinatorially from splitting the spines. In Lemma 3.1 of \cite{AB00} they appeared from Fa\`a di Bruno's differentiation formula.
	\end{rem}
	
	We need the following elementary lemma before we can complete our proof.
	
	\begin{lemma}\label{momentlem}
	For any non-negative integer-valued random variable $Y$, and any integers $a\geq b\geq1$,
	\[\E[Y^a]\E[Y] \geq \E[Y^b]\E[Y^{a-b+1}].\]
	\end{lemma}
	
	\begin{proof}
	Assume without loss of generality that $b\geq a/2$. Note that for any two positive integers $j$ and $k$,
	\begin{align*}
	j^a k + j k^a - j^b k^{a-b+1} - j^{a-b+1} k^b &= jk(j-k)^2\Big(j^{a-3} + 2j^{a-4}k + 3j^{a-5}k^2 + \ldots\\
	&\hspace{10mm} + (a-b)j^{b-2} k^{a-b-1} + (a-b-1)j^{b-3}k^{a-b} + \ldots\\
	&\hspace{20mm} + 2jk^{a-4} + k^{a-3}\Big)\\
	&\geq 0.
	\end{align*}
	Thus
	\begin{align*}
	&\E[Y^a]\E[Y] - \E[Y^b]\E[Y^{a-b+1}]\\
	&= \sum_{j\geq1}j^a \P(Y=j)\sum_{k\geq1}k \P(Y=k) - \sum_{j\geq1}j^b \P(Y=j)\sum_{k\geq1}k^{a-b+1} \P(Y=k)\\
	&=\sum_{j\geq1}\sum_{k>j} (j^a k + k^a j - j^b k^{a-b+1} - j^{a-b+1} k^b)\P(Y=j)\P(Y=k)\\
	&\geq0
	\end{align*}
	as required.
	\end{proof}
	
	We can now finish the proof of the main result.

	\begin{proof}[Proof of Theorem 1 for $M^k$]
	Case i) follows just as for $M^1$, applying dominated convergence to the $\Q^k$-expectation in Lemma \ref{many_to_few}. Note that if $T$ is the first split time of the $k$ spines (as in Lemma \ref{highermoments}) then $e^{(m_k-1)\int_0^T \ind_0(\xi_s)ds}$ is stochastically dominated by $e^{(m_k-1)\tau}$ where $\tau$ is an exponential random variable of parameter $m_k - m_1$; this allows us to construct the required dominating random variable.
	
	\smallskip
	For case ii), using Lemmas \ref{l} and \ref{highermoments}  we find the lower bound
	\begin{align}\label{hh}
		M^k(t,x) \geq  1+C\int_0^tp_s(x,0)M^{k-1}(t-s,0)ds\geq C\int_0^t p_s(x,0) M^{k-1}(t-s,0) ds.
	\end{align}
	An upper bound can be obtained by additionally using Lemma \ref{momentlem} (to reduce $g_k$ to the leading term $M^1 M^{k-1}$) to obtain
	\begin{align}\label{ll}
	M^k(t,x) \leq 1 + C\int_0^t p_s(x,0) M^{k-1}(t-s,0) ds.
	\end{align}
	Using inductively the lower bound (\ref{hh}) and furthermore the iteration
	\begin{align}\label{cal}\begin{split}
	&\int_0^t \P_x(X_{s_1}=0) \int_0^{t-s_1}\P_0(X_{s_2}=0)\ldots \int_0^{t-s_1-\ldots-s_{k-2}} \P_0(X_{s_{k-1}}=0) ds_{k-1} \ldots ds_2 ds_1\\
	&=\int_0^t \int_{s_1}^t \ldots \int_{s_{k-2}}^t \P_x(X_{s_1}=0, X_{s_2}=0,\ldots,X_{s_{k-1}}=0) ds_{k-1} \ldots ds_2 ds_1\\
	&= \frac{1}{(k-1)!}\int_0^t\int_0^t\ldots\int_0^t \P_x(X_{s_1}=0,X_{s_2}=0,\ldots,X_{s_{k-1}}=0)ds_{k-1}\ldots ds_2 ds_1\\
	&= \frac{1}{(k-1)!}\E_x\left[\left(\int_0^t\ind_{\{X_s=0\}}ds\right)^{k-1}\right]\\
  &=\frac{1}{(k-1)!}\E_x\left[L_t(0)^{k-1}\right]
\end{split}
	\end{align}
	we see that $M^k(t,x)$ goes to infinity if $\mathcal A$ is recurrent and to a constant if $\mathcal A$ is transient. This implies that the additional summand $1$ in (\ref{ll}) can be omitted asymptotically in both cases. The claim follows.

	\smallskip
  The lower bound of case iii)a) follows by the same argument as for case ii), and the upper bound is a straightforward induction using Lemmas \ref{highermoments} and \ref{momentlem}. The cases iii)b) and c) also follow from Lemma \ref{highermoments} and induction based on the asymptotics for $M^1$.
	\end{proof}

\section{Acknowledgements}
	LD would like to thank Martin Kolb for drawing his attention to \cite{ABY98} and Andreas Kyprianou for his invitation to the Bath-Paris workshop on branching processes, where he learnt of the many-to-few lemma from MR. The authors would also like to thank Piotr Milos for checking an earlier draft, and a referee for pointing out several relevant articles.

\end{document}